\documentclass[english]{article}
\usepackage{amsmath}
\usepackage{amsfonts}
\usepackage{amssymb}

\usepackage{amsmath}
\usepackage{manfnt}
\usepackage{amssymb}
\usepackage{graphicx}
\usepackage{geometry}
\usepackage{amsthm}
\usepackage[dvipsnames]{xcolor}
\usepackage[utf8]{inputenc}
\usepackage{pst-node}
\usepackage{tikz-cd} 
\usepackage{hyperref}
\usepackage{pdfpages}
\usepackage[backend=biber]{biblatex}
\usepackage[symbols,nogroupskip,nonumberlist,sort=use]{glossaries-extra}

\makenoidxglossaries

\glsxtrnewsymbol[description={The topological anti-dual of a complex Hilbert space $H$}]{$H'$}{$H'$}
\glsxtrnewsymbol[description={Any realization of the anti-dual space of $H$}]{$H^*$}{$H^*$}
\glsxtrnewsymbol[description={The Lebesgue-Bochner space of $L^p$ functions from $(0,T)$ to $H$}]{$L^p(0,T;B)$}{$L^p(0,T;H)$}
\glsxtrnewsymbol[description={The Sobolev-Bochner space of $H^k$ functions from $(0,T)$ to $H$}]{$H^k(0,T;B)$}{$H^k(0,T;H)$}
\glsxtrnewsymbol[description={The space of $H^1(a,b;H)$-functions vanishing at $a$}]{$H_{(a)}^1(a,b;H)$}{$H_{(a)}^1(a,b;H)$}
\glsxtrnewsymbol[description={The space of $H^1(a,b;H)$-functions vanishing at $b$}]{$H_{(b)}^1(a,b;H)$}{$H_{(b)}^1(a,b;H)$}
\glsxtrnewsymbol[description={The space of $C([a,b];H)$-functions vanishing at $a$}]{$C_{(a)}([a,b];H)$}{$C_{(a)}([a,b];H)$}
\glsxtrnewsymbol[description={The space of $C([a,b];H)$-functions vanishing at $b$}]{$C_{(b)}([a,b];H)$}{$C_{(b)}([a,b];H)$}
\glsxtrnewsymbol[description={The space of $H$-valued distributions on $(0,T)$}]{gg}{$\mathcal{D}'(0,T;H)$}
\glsxtrnewsymbol[description={The set of linear and continuous operators from $X$ to 
$Y$}]{dd}{$\mathcal{L}_c(X;Y)$}
\glsxtrnewsymbol[description={The duality pairing between $X$ and $Y$}]{ee}{$\langle \cdot , \cdot \rangle_{X,Y}$}
\glsxtrnewsymbol[description={The scalar product of $H$}]{yy}{$( \cdot , \cdot )_H$}
\glsxtrnewsymbol[description={$\exists c > 0,\quad \forall x,\quad f(x) \leq c g(x)$}]{zz}{$f(x) \lesssim g(x)$}
\glsxtrnewsymbol[description={The Dirac mass at $x_0$, seen as an operator}]{tt}{$\delta_{x_0}$}

\graphicspath{{graphics}}

\newcommand{\opSpan}{\operatorname{Span}}

\newcommand{\opker}{\operatorname{ker}}

\newcommand{\opRange}{\operatorname{Range}}

\newcommand{\opsupp}{\operatorname{supp}}

\newcommand{\opsgn}{\operatorname{sgn}}

\newcommand\numberthis{\addtocounter{equation}{1}\tag{\theequation}}

\newtheorem{Theorem}{Theorem}[section]
\newtheorem{Proposition}[Theorem]{Proposition}
\newtheorem{Remark}[Theorem]{Remark}
\newtheorem{Definition}[Theorem]{Definition}

\newtheorem{Example}[Theorem]{Example}

\newtheorem{Corollary}[Theorem]{Corollary}

\title{On the control of LTI systems with rough control laws}
\author{Lucas Davron\footnote{CEREMADE, Université Paris-Dauphine, CNRS UMR 7534, Université PSL, 75016
Paris, France, \\ \href{mailto:davron@ceremade.dauphine.fr}{davron@ceremade.dauphine.fr}}}

\DeclareUnicodeCharacter{2212}{-}

\begin{document}
\maketitle
\begin{abstract}
The theory of linear time invariant systems is well established and allows, among other things, to formulate and solve control problems in finite time. In this context the control laws are typically taken in a space of the form $L^p(0,T;U)$. In this paper we consider the possibility of taking control laws in $(H^1(0,T;U))^*$, which induces non-trivial issues. We overcome these difficulties by adapting the functional setting, notably by considering a generalized final state for the systems under consideration. In addition we collect time regularity properties and we pretend that in general it is not possible to consider control laws in $H^{-1}(0,T;U)$. Then, we apply our results to propose an interpretation of the infinite order of defect for an observability inequality, in terms of controllability properties. 
\end{abstract}
\section{Introduction}
In this paper we consider an arbitrary linear time invariant (LTI) system
\begin{equation}\label{system_AB}
\left\lbrace \begin{array}{c c c}
\dot{z}(t) &=& Az(t) + Bu(t),\\
z(0) &=& z_0, \\
\end{array} \right.
\end{equation}
where $z(t)$ is the state of the system at time $t$, and $u(t)$ is the control exerced on the system at time $t$. We will frequently use the notation $\Sigma(A,B)$ to refer to the control system whose evolution is governed by \eqref{system_AB}. Such systems have been extensively studied as they model a wide range of control problems, notably partial differential equations (PDEs), and we refer to \cite{Coron}, \cite{curtain_zwart}, \cite{tucwei} and \cite{weiss_admissibility} for a discussion on these LTI systems and their applications. \newline
\newline
We will make the following assumptions: 
\begin{itemize}
\item The function $z(\cdot)$ takes its values in $X$ (the state space) which is a Hilbert\footnote{As is customary in the LTI formalism, all the vector spaces will be assumed to be complex. However, complex numbers will not be used until the last section of this paper, and the reader may as well forget about the complex structure if it is more convenient.} space; 
\item The function $u(\cdot)$ takes its values in another Hilbert space $U$ (the input space);
\item $A$ is the generator of a $C_0$ semigroup $(S_t)_{t\geq 0}$ on $X$;
\item $B \in \mathcal{L}_c(U;D(A^*)^*)$ is unbounded and 2-admissible\footnote{The space $D(A^*)^*$ is the anti-dual space of $D(A^*)$ with respect to the pivot $X$, which is often denoted by $X_{-1}$, see \cite[Section 2.9]{tucwei} and \cite[Chapitre III]{aubin}. The admissibility of such an operator $B$ was introduced in the seminal paper \cite{weiss_admissibility} to which we refer the reader for more details.}.
\end{itemize}
Such hypotheses are commonly used to derive the well-posedness of \eqref{system_AB} as well as a control theory for $\Sigma(A,B)$. Indeed, fix a finite horizon time $0 < T < \infty$. Under the standing assumptions the problem \eqref{system_AB} is well-posed so that, for any $z_0 \in X$ and $u \in L^2(0,T;U)$, it has a unique solution $z(\cdot) \in C([0,T] ; X)$ (see \textit{e.g.} \cite[Theorem 2.37]{Coron}). Note that from a differential equation perspective it is important to specify what is the concept of solution that is used to obtain well-posedness. In \cite{Coron} the concept of solution is that of transpositions, whereas in \cite{tucwei} and \cite{weiss_admissibility} the authors consider strong solutions. From the perspective of LTI systems, the solution concept is less important as it shall lead to the same input-to-state map 
\[
\left\lbrace \begin{array}{ccc}
(z_0,u) & \longmapsto & z(\cdot) \\
X \times L^2(0,T;U) & \longrightarrow & C([0,T] ; X)
\end{array}\right. ,
\]
which is given by the Duhamel formula 
\[
z(t) = S_tz_0 + \int_0^tS_{t-s}Bu(s)ds,
\]
for any $t \in [0,T]$, $z_0 \in X$ and $u \in L^2(0,T;U)$. Using the above formula one may derive the following fundamental result of control theory, which is known as the duality between observability and controllability.  
\begin{Definition}
The system $\Sigma(A,B)$ is said null controllable (at time $T$) when 
\[
\forall z_0 \in X,\quad \exists u \in L^2(0,T;U),\quad z(T) = 0.
\]
\end{Definition}
\begin{Theorem}{\cite[Theorem 2.44]{Coron}}
The system $\Sigma(A,B)$ is null controllable (at time $T$) if and only if 
\begin{equation}\label{eq:observability}
\exists c > 0,\quad \forall \varphi \in X,\quad \|S_T^* \varphi\|_X^2 \leq c \int_0^T \| B^*S_t^*\varphi\|_U^2 dt.
\end{equation}
\end{Theorem}
The inequality \eqref{eq:observability} above is commonly referred to as the ``final time observability of the adjoint system'' and for brevity we will refer to it as the ``observability inequality". \newline
\newline
In this paper we consider the possibility of taking $u$ in $\mathcal{U}$, where 
\[
\mathcal{U} = (H^1(0,T;U))^*,\quad \mathrm{or}~~ \mathcal{U} = H^{-1}(0,T;U) := (H^1_0(0,T;U))^*.
\]
Let us motivate such an undertaking. Firstly, enlarging the space $\mathcal{U}$ may be useful because the larger $\mathcal{U}$ is, the easier it is to prove the null controllability of the system $\Sigma(A,B)$. \par 
Secondly, such control laws have already been considered in specific circumstances, notably in \cite[Section 3]{lions_contro} where the author controls several types of wave equations, from the boundary, choosing control laws in spaces closely related to those $\mathcal{U}$ defined above. This should be compared with the general fact that hyperbolic systems generally do not enjoy any regularization property, and may even preserve the regularity of the initial and boundary data, both in space and in time. As a consequence, in order to control a hyperbolic system one may expect to need to use irregular control laws. For further examples we refer to \cite{Cavalcanti} and \cite{LT_wave}, which both deal with wave equations, and also to \cite{bao} which studies another type of second order hyperbolic PDE. We also mention \cite{Rebarber_Zwart}, where the authors consider the stabilization of LTI systems taking control laws in the distribution space 
\[
\mathcal{D}_0'(U) = \{ u \in \mathcal{D}'(\mathbb{R};U) : \opsupp u \subset [0,\infty) \}.
\]
Although such control laws are essentially more general than the one we consider here, our approach is different as we are concerned with finite time controllability. \par 
Thirdly, one may wonder how to translate the observability inequality \eqref{eq:observability} when the right hand side is replaced by the $H^1$-norm:
\[
\|S_T^* \varphi\|_X^2 \leq c \| B^*S_t^*\varphi\|_{H^1(0,T;U)}^2,
\]
in terms of controllability properties. In this view it is desirable to have a definition of what is a control system with control laws in $\mathcal{U}$ as above. \newline
\newline 
The key issue will be to \textit{define} such a control system. The main difficulties can be summed-up as:
\begin{itemize}
\item The state trajectory may be discontinuous with respect to time, hence the final state $z(T)$ is \textit{a priori} ill-defined, and one has to consider a generalized final state. 
\item It may happen that the state trajectory $z(\cdot)$ is continuous in time, but the generalized final state cannot be deduced from the curve $z(\cdot)$. 
\item When $\mathcal{U} = H^{-1}(0,T;U)$, the generalized final state may be ill-defined, even in a very weak sense. 
\end{itemize} 
We now state our main results. We begin by giving a natural sense to the (generalized) final state $z(T)$ when $u(\cdot) \in (H^1(0,T;U))^*$. Until further notice we shall assume that in \eqref{system_AB}, $z_0=0$, which is not a restriction by the superposition principle. Accordingly, for all $u(\cdot) \in L^2(0,T;U)$, we let $z(\cdot)$ be the solution of \eqref{system_AB} with $z_0 = 0$. 
\begin{Theorem}\label{theo:extension_time_T}
The map $u(\cdot) \mapsto (z(\cdot) , z(T))$ has a unique linear and continuous extension from $(H^1(0,T;U))^*$ to $L^2(0,T;D(A^{*2})^*) \times D(A^*)^*$.
\end{Theorem}
In the above statement, the space $D(A^{*2})^*$ is the anti-dual space of $D(A^{*2})$ (the domain of $A^{*2} = (A^*)^2$) with respect to the pivot $X$. We will also consider more general control laws, see Theorem \eqref{theo:extension_precise}. \par
For further reference let us introduce the operators 
\[
\Xi : \left\lbrace \begin{array}{ccc}
u & \longmapsto & z(\cdot) \\
L^2(0,T;U) & \longrightarrow & C([0,T] ; X)
\end{array} \right. , \quad
\Xi_T : \left\lbrace \begin{array}{ccc}
u & \longmapsto & z(T) \\
L^2(0,T;U) & \longrightarrow & X
\end{array} \right. ,
\]
which will be called respectively the state curve and the final state. The above Theorem can be equivalently formulated as these operators having extensions, respectively from $(H^1(0,T;U))^*$ to $L^2(0,T;D(A^{*2})^*)$ and from $(H^1(0,T;U))^*$ to $D(A^*)^*$. We will keep the same notations for the extensions.
\begin{Example}\label{ex:toy_model}
In order to illustrate the situation, let us consider the toy-model 
\begin{equation} \label{eq:toy_model}
\left\lbrace \begin{array}{cccc}
\dot{x}(t) &=& f(t),& 0 \leq t \leq T, \\
x(0) &=& 0,
\end{array}\right.
\end{equation}
where $0 < T < \infty$ is a fixed final time and $x,f:(0,T) \rightarrow \mathbb{R}$ are functions. If $f \in L^1(0,T)$, then the solution of \eqref{eq:toy_model} is given by
\[
\forall t \in [0,T],\quad x(t) = \int_0^t f(s)ds.
\]
For this system, the above Theorem simply states that the operators
\[
\Xi_T : f(\cdot) \mapsto \int_0^T f(s) ds,\quad \Xi : f(\cdot) \mapsto \left[ t \mapsto \int_0^t f(s)ds \right],
\]
respectively have $\mathcal{L}_c((H^1(0,T))^* ; \mathbb{R})$ and $\mathcal{L}_c((H^1(0,T))^* ; L^2(0,T))$ extensions, where $\mathcal{L}_c$ stands for the set of linear and continuous maps. Consider the source term $f = \delta_T \in (H^1(0,T))^*$, it is elementary to check that $\Xi f$ is almost everywhere vanishing, while
\[
\Xi_T f = 1.
\]
Therefore, the relation
\[
\Xi_T f = (\Xi f)(T), \quad f \in L^1(0,T),
\]
does not hold anymore when $f \in (H^1(0,T))^*$, even if $\Xi f \in C([0,T])$. 
\end{Example}
We see on this basic example that in general, one cannot compute (or even estimate) the generalized final state from the generalized state curve. \newline
\newline
Next we claim that in full generality, the generalized state curve $z(\cdot)$ cannot be better than $L^2$ in time.  
\begin{Proposition}\label{prop:sharp_reg}
Assume that $B \neq 0$. Then there exists $u \in (H^1(0,T;U))^*$ such that
\[
\forall 0 \leq t_0 < t_1 \leq T,\quad \forall p > 2,\quad z(\cdot) \notin L^p(t_0,t_1; D(A^{*\infty})^*).
\]
\end{Proposition}
Especially, even after weakening the topology of $X$ to become the universal extrapolation space $D(A^{*\infty})^*$, the state trajectory $z(\cdot)$ may be discontinuous with respect to time. We complete this negative result by a method to get time continuity for the state curve. We consider
\[
\opker B^* = \{ \varphi \in D(A^*) : B^* \varphi = 0 \},
\]
which is a Hilbert space when endowed with the $D(A^*)$-norm. Assuming that it is a dense subset of $X$, we may consider $(\opker B^*)^*$ its anti-dual space with respect to the pivot $X$. 
\begin{Proposition}\label{prop:time_continuity}
Assume that $\opker B^*$ is dense in $X$. Then, the mapping $u \mapsto z(\cdot)$ extends linearly and continuously from $(H^1(0,T;U))^*$ to $C([0,T] ; (\opker B^*)^*)$.
\end{Proposition}
In Proposition \eqref{prop:time_reg_arbitrary} we will strengten this result, giving hypotheses ensuring that the state curve is of class $C^N$ with respect to time. Note carefully that Proposition \eqref{prop:time_continuity} does not contradict Proposition \eqref{prop:sharp_reg} above. Indeed if $z(\cdot) \in C([0,T] ; (\opker B^*)^*)$, to infer that that $z(\cdot) \in C([0,T] ; D(A^{*\infty})^*)$ one needs a natural embeding $(\opker B^*)^* \hookrightarrow D(A^{*\infty})^*$. The only reasonable way to get such an inclusion would be to have $D(A^{*\infty})$ being a dense subset of $\opker B^*$, which is the case only if $B=0$\footnote{If $D(A^{*\infty})$ is dense in $\opker B^*$, being $D(A^{*\infty})$ dense in $D(A^*)$ we get $\opker B^*$ dense in $D(A^*)$. Because  $\opker B^*$ is closed in $D(A^*)$, this means $D(A^*) = \opker B^*$, so that $B^*=0$, and $B=0$.}. \newline
\newline
Further, we consider control laws $u$ in $H^{-1}(0,T;U)$, and we will see that it is not obvious to make sense of the generalized final state for such $u$. Consider again the toy-model \eqref{eq:toy_model}, since 
	\[
	\forall f \in L^2(0,T),\quad x(T) = \int_0^T f(t)dt = \langle f , 1 \rangle,
	\]
we clearly see that it is not possible to extend the linear form $f \mapsto x(T)$ to the whole of $H^{-1}(0,T)$. Indeed if this was possible, the evaluation against the function $1$ would be in $(H^1_0(0,T))''$, and since $H^1_0(0,T)$ is a reflexive space we would get $1 \in H^1_0(0,T)$, a contradiction. In Section \eqref{sec:counter_h-1} we will propose a more convincing example of system $\Sigma(A,B)$, in infinite dimension, where this obstruction is blatant. Conversely, we will also propose a functional setting for which it makes sense to take $u \in H^{-1}(0,T;U)$. \newline
\newline
After having understood that the right functional setting is to take $u \in (H^1(0,T;U))^*$ and to consider the generalized final state, we obtain almost for free the duality between controllability and observability.
\begin{Proposition}The system $\Sigma(A,B)$ is null controllable (at time $T$) with initial data in $X$ and control laws in $(H^1(0,T;U))^*$ if and only if there is a constant $c >0$ such that for all $\varphi \in D(A^*)$, there holds 
\[
\| S_T^* \varphi \|_{X} \leq c  \| B^* S_t^* \varphi \|_{H^1(0,T;U)}.
\]
\end{Proposition}
As for the other results, this Proposition will be extended to encompass various choices for the spaces of initial data and control laws, see Proposition \eqref{prop:contro_general_rigorous}. \newline
\newline
As an application of these results, we interpret a result of \cite{Zhang_Zuazua}, where the authors study the fluid-structure model
\begin{equation} \label{eq_zhang_zuazua}
\left\lbrace \begin{array}{cccc}
u_t &=& u_{xx}, & 0 < x < 1, \\
v_{tt} &=& v_{xx}, & -1 < x < 0,\\
u(t,1) &=& g_1(t), \\
v(t,-1) &=& 0, \\
u(t,0) &=& v(t,0), \\
u_x(t,0) &=& v_x(t,0), \\
\end{array}\right.
\end{equation}
where $g_1(t)$ is the control. They prove that this system is not null controllable, showing that the observability inequality \eqref{eq:observability} is not satisfied. In fact they show a stronger result, namely that this observability inequality has an infinite order of defect, in the sense that the weakened observability inequality 
\begin{equation}\label{eq_weakened_obs_N}
\|S_T^* \varphi \|_{X}^2 \lesssim \| B^*S_t^* \varphi \|_{H^N(0,T;U)}^2,
\end{equation}
never holds for any $N \in \mathbb{N}$. They then suggest that the system \eqref{eq_zhang_zuazua} is not null controllable with controls in $H^{-s}(0,T)$, for all $s \geq 0$ (see \cite{Zhang_Zuazua}, after the proof of Theorem 4.2). In view of our results, we propose here a justified and more natural interpretation of the above infinite defect in the observability inequality, in the sense that we consider control laws belonging to what we believe to be the ``right" class of irregular controls. To this aim consider $X_h$ the subspace of $X$ generated by the hyperbolic eigenvectors of $A$, and $P_h$ the associated spectral projection, which will be both introduced in Section \eqref{Section:application}. 
\begin{Proposition}\label{prop:zz_nc} Let $N \in \mathbb{N}$ and $0 < T < \infty$ arbitrary. Then the system \eqref{eq_zhang_zuazua} is not null controllable (at time $T$), with output operator $P_h$, initial states in $X_h \cap D(A^N)$ and control laws in $(H^N(0,T))^*$. 
\end{Proposition} 
We emphasize that in Proposition \eqref{prop:zz_nc} above, considering only the output operator $P_h$ and initial conditions in $X_h$ would amount to project the system $\Sigma(A,B)$ on the $(S_t)_{t\geq 0}$-invariant subspace $X_h$, see \textit{e.g.} \cite[Section 4]{jacob_zwart}. In this view, the most important point here is that we are able to consider control laws in $(H^N(0,T))^*$, while further restricting initial conditions to be taken in $D(A^N)$ is slightly tedious for admissibility considerations but will not induce any difficulty. \newline
\newline
The rest of the paper goes as follows. In Section 2 we construct several spaces that will be useful to extend the operators under consideration. In Section 3 we extend the final state and the input-to-state map. In Section 4 we consider control laws in $H^{-1}$ as well as time regularity for the state curve. In Section 5 we establish the duality between controllability and observability, with extended inputs. In Section 6 we apply our results to the fluid-structure model presented above. \hypertarget{notations}{}
\printunsrtglossary[type=symbols,style=long,title={Notations}]
\section{Construction of the Sobolev towers}\label{sec_sobolev_towers}
Given an operator $G$ which is the generator of a $C_0$ semigroup on a Hilbert space $H$, there is a classical construction of a family of interpolation/extrapolation spaces for $G$ giving rise to the so-called abstract Sobolev tower associated to $G$. Let us briefly recall its construction, for more details we refer to \cite[Section 2.10]{tucwei} and \cite[Section II.5]{engel}. For all $n \in \mathbb{N}$, the space $H_n$ is the domain $D(G^n)$ of $G^n$ endowed with the graph norm 
\[
\|z\|_{H_n}^2 = \|z\|_H^2 + \|G^nz \|_H^2.
\]
Note that by convention, $G^0 = 1$, so that $H_0 = H$. This definitions makes $H_n$ a Hilbert space as $G^n$ is closed (see \cite[Theorem 7, Section 9, Chapter 7]{dunford_schwartz}). Further we define for every $n \in \mathbb{N}^*$ the space $H_{-n}$ by
\[
H_{-n} = D(G^{*n})^*,
\]
that is the anti-dual space of $D(G^{*n})$ with respect to the pivot space $H$. This defines a family $(H_n)_{n \in \mathbb{Z}}$ of Hilbert space such that for all $n \in \mathbb{Z}$, the inclusion $H_{n+1} \subset H_n$ is continuous and dense. We represent this Sobolev tower by 
\[
... \supset D(G^{*2})^* \supset D(G^*)^* \supset H \supset D(G) \supset D(G^2) \supset ...
\]
These spaces are convenient, and for instance for all $n \in \mathbb{Z}$ the operator $G$ generates a $C_0$ semigroup on $H_n$ (see \cite[Proposition 2.10.4]{tucwei}). \par 
In order to extend the operators $\Xi$ and $\Xi_T$ (respectively the state curve and the final state), we will use slightly different families of Hilbert spaces. These families can be thought of as modifications of the standard Sobolev tower, as they will serve the same purpose of providing a scale of regularity. For any $N \in \mathbb{N}$, we let 
\[
X_N = D(A^{*N}),
\]
endowed of the graph norm, which is a Hilbert space. As before we have
\[
X_0 = X,
\]
and for all $N \in \mathbb{N}^*$ we define 
\[
X_{-N} = (X_N)^*,
\]
where $(X_N)^*$ is the anti-dual space of $X_N$ with respect to $X$. The family of Hilbert spaces $(X_N)_{N \in \mathbb{Z}}$ has the following natural properties:
\begin{itemize}
\item For all $N_1 \leq N_2 \in \mathbb{Z}$, the inclusion $X_{N_2} \subset X_{N_1}$ is dense and continuous.
\item For all $N \in \mathbb{Z}$, there holds $(X_N)^* = X_{-N}$.
\item For all $N_1 \leq N_2 \in \mathbb{N}$, the successive duality brackets $\langle \cdot , \cdot \rangle_{X_{-N_2} , X_{N_2}}$ and $\langle \cdot , \cdot \rangle_{X_{-N_1} , X_{N_1}}$ are compatible in the sense that
\[
\forall z \in X_{-N_1},\quad \forall \varphi \in X_{N_2},\quad \langle z , \varphi \rangle_{X_{-N_1} , X_{N_1}} = \langle z , \varphi \rangle_{X_{-N_2} , X_{N_2}}.
\]
\end{itemize}
Observe that $(X_N)_{N \in \mathbb{Z}}$ is not one of the standard Sobolev towers associated respectively to $A$ or $A^*$. This can be seen by writting $(X_N)_{N \in \mathbb{Z}}$ as 
\[
... \supset D(A^{*2})^*  \supset D(A^*)^* \supset X \supset D(A^*) \supset D(A^{*2}) \supset ...
\]
This choice of Sobolev tower is determined by the two properties
\[
\forall N \in \mathbb{N},\quad X_N = D(A^{*N}),
\]
and
\[
\forall N \in \mathbb{Z},\quad (X_N)^* = X_{-N},
\]
which are essential for our method of extension. Indeed, our results are shown by first obtaining a regularity property for the adjoint operator (\textit{ergo} the first property) and second by deducing an extension result for the primal operator (\textit{ergo} the second property). \par 
The Sobolev tower used for the control laws is very similar. For any $M \in \mathbb{N}$ we let 
\[
\mathcal{U}_M = H^M(0,T;U),
\]
endowed of its standard Hilbert space structure. For $M \in \mathbb{N}^*$, we let 
\[
\mathcal{U}_{-M} = \left(H^M(0,T;U) \right)^*,
\]
where the anti-dual space is taken with respect to the pivot $L^2(0,T;U)$. These spaces are such that the three aforementioned natural properties of $(X_N)_{N \in \mathbb{Z}}$ are also satisfied by $(\mathcal{U}_M)_{M \in \mathbb{Z}}$. \par 
Note that in general for $N \geq 1$, 
\begin{itemize}
\item $A^*$ is a generator on $X_{N}$, but may fail to be a generator on $X_{-N}$;
\item $A$ is a generator on $X_{-N}$, but may fail to be a generator on $X_N$.
\end{itemize}
\section{Extensions}
From now on, we fix $U$ and $X$ two Hilbert spaces. Let $A$ be the infinitesimal generator of a $C_0$ semigroup $(S_t)_{t\geq 0}$ on $X$ and $B \in \mathcal{L}_c(U;D(A^*)^*)$ be 2-admissible. We fix a finite time horizon $0 < T < \infty$, as well as the Sobolev towers $(X_N)_{N \in \mathbb{Z}}$ and $(\mathcal{U}_M)_{M \in \mathbb{Z}}$ constructed in Section \eqref{sec_sobolev_towers}. This Section is devoted to the proof of the following result. 
\begin{Theorem}\label{theo:extension_precise}
For any $N \in \mathbb{N}^*$, the mapping $(z_0,u) \mapsto z(T)$ has a unique linear and continuous extension from $X_{-N} \times \mathcal{U}_{-N}$ to $X_{-N}$. Furthermore, the mapping $(z_0,u) \mapsto z(\cdot)$ has a unique linear and continuous extension from $X_{-2} \times \mathcal{U}_{-1}$ to $L^2(0,T;X_{-2})$. 
\end{Theorem}
\begin{Remark}
Although it is possible to extend the state curve for control laws in $\mathcal{U}_{-N}$ when $N > 1$, using the same techniques as in the case $N=1$, we will not state or prove any result in this direction. This is because on the one hand, in such a general setting the state curve might be irrelevant as illustrated in Example \eqref{ex:toy_model}, and on the other hand the proof of such an extension is rather lengthy and technical but not much instructive.
\end{Remark}
Let us introduce again the input-to-state and final state operators, but now without the hypothesis $z_0 = 0$,
\begin{equation}\label{eq:def_Xi}
\Xi_T : \left\lbrace \begin{array}{ccc}
X \times L^2(0,T;U) & \longrightarrow & X \\
(z_0,u) & \longmapsto & z(T)
\end{array}\right.,\quad 
\Xi : \left\lbrace \begin{array}{ccc}
X \times L^2(0,T;U) & \longrightarrow & C([0,T];X) \\
(z_0,u) & \longmapsto & z(\cdot)
\end{array}\right..
\end{equation}
The Theorem above states that these operators may be extended. In view of Duhamel's formula 
\[
z(t) = S_tz_0 + \int_0^t S_{t-s}Bu(s)ds,
\]
one may deal with the cases $z_0 = 0$ and $u=0$ separately. Note that the case $u=0$ is trivial as $(S_t)_{t \geq 0}$ is also a $C_0$ semigroup on $X_{-N}$ for all $N \in \mathbb{N}^*$ (see \cite[Proposition 2.10.4]{tucwei}), so that 
\[
\left[ t \mapsto S_tz_0 \right] \in C([0,\infty) ; X_{-N}),
\]
for all $z_0 \in X_{-N}$. Therefore, we are left to extend the operators 
\[
F_T : \left\lbrace \begin{array}{ccc}
L^2(0,T;U) & \longrightarrow & X \\
u & \longmapsto & z(T)
\end{array}\right., \quad 
F : \left\lbrace \begin{array}{ccc}
L^2(0,T;U) & \longrightarrow & L^2(0,T;X)\\
u & \longmapsto & z(\cdot)
\end{array}\right.,
\]
where $z(\cdot)$ is the solution of \eqref{system_AB} with $z_0=0$. 
\subsection{The final state}
In this Subsection we deal with the extension of $F_T$. Because the change of variable $t = T-s$ induces an isomorphism on all the spaces that we will consider, we may as well assume that $F_T$ is defined by 
\[
F_T u = \int_0^T S_tBu(t)dt.
\]
\begin{Proposition}\label{prop:extension_F_T}
For all $N \in \mathbb{N}$, the operator $F_T$ has a unique $\mathcal{L}_c(\mathcal{U}_{-N} ; X_{-N})$ extension. 
\end{Proposition}
\begin{proof}
\underline{Step 1:} We first consider the adjoint of $F_T^*$ of $F_T$, which writes as
\[
F_T^* : \left\lbrace \begin{array}{ccc}
\varphi & \longmapsto & \left[ t \mapsto B^*S_t^*\varphi \right] \\
X & \longrightarrow & L^2(0,T;U)
\end{array}\right. ,
\]
and we show that it is of class $\mathcal{L}_c(X_{N} ; \mathcal{U}_N)$ for all $N \in \mathbb{N}$. To this aim let $N \in \mathbb{N}$ and fix $\varphi \in X_{N+1}$. We first claim that  
\begin{equation}\label{eq:reg_Q}
\left[ t \mapsto B^*S_t^*\varphi \right] \in C^N([0,\infty) ; U).
\end{equation}
Indeed, being $(S_t^*)_{t\geq 0}$ a $C_0$ semigroup on $X$ with generator $A^*$, we automatically have 
\[ \left[ t \mapsto S_t^*\varphi \right] \in C^{N+1}([0,\infty) ; X),\quad \frac{d^k}{dt^k}S_t^*\varphi = S_t^*A^{*k}\varphi,\quad \forall k = 0,..., N+1, \]
in view of  
\[
\varphi \in X_{N+1} = D(A^{*(N+1)}).
\]
We therefore deduce that 
\[ \left[ t \mapsto S_t^*\varphi \right] \in C^N([0,\infty) ; D(A^*)),\]
so that being $B^* \in \mathcal{L}_c(D(A^*) ; U)$ we obtain the claimed regularity \eqref{eq:reg_Q}. \par 
Second, we show that $F_T^*$ is $\mathcal{L}_c(X_N;\mathcal{U}_N)$. Because $X_{N+1}$ is a dense linear subspace of $X_N$ and $\mathcal{U}_N$ is complete, it is enough to show that 
\[
\exists c > 0,\quad \forall \varphi \in X_{N+1} ,\quad \| F_T^*\varphi \|_{\mathcal{U}_N} \leq c \|\varphi\|_{X_N}.
\]
To this aim we estimate 
\begin{align*}
\|F_T^* \varphi \|_{\mathcal{U}_N}^2 &= \sum_{k=0}^N \left\| \frac{d^k}{dt^k} B^*S_t^*\varphi \right\|_{L^2(0,T;U)}^2 \\
&= \sum_{k=0}^N \left\| B^*S_t^*(A^*)^k\varphi \right\|_{L^2(0,T;U)}^2 \\
&\leq \sum_{k=0}^N C_{\mathrm{adm}}^2 \| (A^*)^k\varphi\|_X^2\\
&\lesssim C_{\mathrm{adm}}^2\| \varphi\|_{X_N}^2,
\end{align*}
where $C_{\mathrm{adm}} > 0$ arises from the admissibility of $B$.\par 
\underline{Step 2:} We use the previously shown property of $F_T^*$ to deduce the desired extension result. \par 
We compute for all $u \in L^2(0,T;U)$ and $\varphi \in X_N$,
\begin{align*}
| \langle F_Tu,\varphi\rangle_{X_{-N},X_N} | &= |(F_Tu,\varphi)_X| \\
&= | (u,F_T^*\varphi)_{L^2(0,T;U)} \\
&=| \langle u , F_T^*\varphi \rangle_{\mathcal{U}_{-N} , \mathcal{U}_N}| \\
&\leq \|u\|_{\mathcal{U}_{-N}} \|F_T^*\varphi\|_{\mathcal{U}_N} \\
&\leq \|u\|_{\mathcal{U}_{-N}}\|F_T^*\|_{\mathcal{L}_c(X_N;\mathcal{U}_N)}\|\varphi\|_{X_N},
\end{align*}
so that 
\[
\|F_Tu\|_{X_{-N}} \leq \|F_T^*\|_{\mathcal{L}_c(X_N;\mathcal{U}_N)}\|u\|_{\mathcal{U}_{-N}},
\]
which brings the claimed extension property for $F_T$ and ends the proof. 
\end{proof}
In fact the above proof yields a slight refinement of the previously shown extension for the final state. Recall that the final state $\Xi_T$ was defined in \eqref{eq:def_Xi}, and that this operator is \textit{a priori} linear and continuous from $X \times L^2(0,T;U)$ to $X$. 
\begin{Corollary}\label{coro:extension_NM}
Let $(N,M) \in \mathbb{Z}^2$. We then have the following: 
\begin{enumerate}
\item If $N,M \geq 0$, the map $\Xi_T$ is $\mathcal{L}_c(X_N \times \mathcal{U}_M ; X)$.
\item If $N < 0$ or $M < 0$, the map $\Xi_T$ has a unique extension which is $\mathcal{L}_c(X_N \times \mathcal{U}_M ; X_{\min(N,M)})$.
\end{enumerate}
\end{Corollary}
We will sum-up this Corollary by saying that, for all $(N,M) \in \mathbb{Z}^2$,
\[
\Xi_T \in \mathcal{L}_c(X_N \times \mathcal{U}_M ; X_{\min(0,N,M)}).
\]
\begin{proof}
In view of the Duhamel formula and since the spaces $(X_N)_{N \in \mathbb{Z}}$ are in non-increasing order, it is enough to show that for all $N \in \mathbb{Z}$, there holds 
\[
S_T \in \mathcal{L}_c(X_N ; X_{\min(0,N)}),\quad F_T \in \mathcal{L}_c(\mathcal{U}_N ; X_{\min(0,N)}).
\]
We will only show the first statement, the second one being similar to prove. Observe that if $N \geq 0$, this re-writes as 
\[
S_T \in \mathcal{L}_c(X_N ; X),
\]
which is indeed true since $S_T \in \mathcal{L}_c(X)$ and the inclusion $X_N \subset X$, for such $N \geq 0$, is continuous. If $N <0$, we are left to show that $S_T \in \mathcal{L}(X_N)$, which is a known fact (see \cite[Proposition 2.10.4]{tucwei}).
\end{proof}
\subsection{The input-to-state map}
We now turn our attention to the extension of the operator
\[
F : \left\lbrace \begin{array}{ccc}
L^2(0,T;U) & \longrightarrow & L^2(0,T;X) \\
u & \longmapsto & z(\cdot)
\end{array}\right..
\]
\subsubsection{Extension by duality}
In this Section we show the following result.
\begin{Theorem}\label{theo:extension_curve}
The operator $F$ has a unique linear and continuous extension from $(H^1(0,T;U))^*$ to $L^2(0,T;X_{-2})$.
\end{Theorem}
Before proving this result we make some comments on the functional space $L^2(0,T;X_{-2})$. Because the inclusion $D(A^{*2}) \subset X$ is continuous and dense, the inclusion
\[
L^2(0,T;D(A^{*2})) \subset L^2(0,T;X),
\]
is also continuous and dense. Taking the adjoint of the latter inclusion, we obtain a continuous inclusion
\[
L^2(0,T;X) \subset (L^2(0,T;D(A^{*2})))^*,
\]
where the right hand side is the anti-dual space of $L^2(0,T;D(A^{*2})^*)$ with respect to the pivot $L^2(0,T;X)$. At this point it is tempting to make the identification 
\begin{equation}\label{eq:dual_L2_Bochner}
(L^2(0,T;D(A^{*2})))^* = L^2(0,T;D(A^{*2})^*).
\end{equation}
In view of \cite[Section 4.1]{diestel_uhl}, for any Hilbert space $H$ the Riesz isomorphism induces a (natural) isomorphism 
\[
L^2(0,T;H') \rightarrow (L^2(0,T;H))',
\]
so that the identification \eqref{eq:dual_L2_Bochner} indeed holds, as $D(A^{*2})$ is a Hilbert space. In other words, the space $L^2(0,T;D(A^{*2})^*)$ is a realization of the anti-dual space of $L^2(0,T;D(A^{*2}))$ with respect to the pivot $L^2(0,T;X)$, with an obvious identification. 
\begin{proof}[Proof of Theorem \eqref{theo:extension_curve}] By duality, if
\begin{equation}\label{eq:reg_state_curve}
F^* \in \mathcal{L}_c(L^2(0,T;D(A^{*2})) ; H^1(0,T;U)),
\end{equation}
then 
\[
F \in \mathcal{L}_c((H^1(0,T;U))^*;(L^2(0,T;D(A^{*2})))^*),
\]
which proves the Theorem in view of \eqref{eq:dual_L2_Bochner}. We are thus left to show that \eqref{eq:reg_state_curve} holds. To do so we acknowledge that, for all $\varphi \in L^2(0,T;D(A^{*2}))$, there holds 
\[
(F^*\varphi)(s) = \int_s^T B^*S_{t-s}^*\varphi(t) dt,
\]
for almost every $s \in (0,T)$. Standard arguments in semigroup theory bring that the above parametrized integral is $C([0,T] ; U)$ and vanishes at $s=T$. Furthermore, the Leibniz rule for the differentiation of parametrized integral yields 
\[
\frac{d}{ds} \int_s^T B^*S_{t-s}^*\varphi(t) dt = -\int_s^T B^*S_{t-s}^*A^*\varphi(t) dt - B^*\varphi(s),\quad \mathrm{in} ~~\mathcal{D}'(0,T;U).
\]
Therefore, \eqref{eq:reg_state_curve} boils down to estimate the $L^2(0,T;U)$-norm of the right hand side in the above formula, in terms of the $L^2(0,T;D(A^{*2}))$-norm of $\varphi$. On the one hand, because $B^* \in \mathcal{L}_c(D(A^*);U)$ we have 
\[
\int_0^T \| B^*\varphi(s) \|_U^2 ds \lesssim \|\varphi\|_{L^2(0,T;D(A^*))}^2 \lesssim \|\varphi\|_{L^2(0,T;D(A^{*2}))}^2.
\]
On the other hand, for all $s \in [0,T]$ we obtain 
\begin{align*}
\left\| \int_s^T B^*S_{t-s}^*A^*\varphi(t) dt \right\|_U &\lesssim \int_s^T \| S_{t-s}^*A^*\varphi(t) \|_{D(A^*)}dt \\
&\lesssim \int_s^T \|A^* \varphi(t) \|_{D(A^*)}dt \\
&\lesssim \|\varphi\|_{L^1(0,T;D(A^{*2}))} \numberthis \label{eq:estimation_Linfty}\\
&\lesssim \|\varphi\|_{L^2(0,T;D(A^{*2}))},
\end{align*}
where the second inequality is from $(S_t^*)_{t\geq 0}$ generating a $C_0$ semigroup on $D(A^*)$. This shows the Theorem
\end{proof}
\subsubsection{Sharp regularity in time}
We show that $L^2$ in time is an optimal regularity in time for $z(\cdot)$, in the sense of the following result. To ease the statement we assume that $z_0 = 0$, but it is clear that the choice of $z_0 \in D(A^{*\infty})^*$ does not matter. 
\begin{Proposition}\label{Prop:optimal_reg_N=1}
Assume that $B \neq 0$. Then, there exists $u \in (H^1(0,T;U))^*$ such that
\[
\forall p >2,\quad \forall 0 \leq t_0 < t_1 \leq T, \quad  z(\cdot) \notin L^p(t_0,t_1;D(A^{*\infty})^*).
\]
\end{Proposition}
In this statement, the integrability condition $z \in L^p(t_0,t_1;D(A^{*\infty})^*)$ refers to the scalar integrability 
\[
\forall \varphi \in D(A^{*\infty}),\quad \left[ t \mapsto \langle z(t),\varphi \rangle \right] \in L^p(t_0,t_1),
\]
where $\langle \cdot , \cdot \rangle$ stands (for instance) for the duality pairing between $X_{-2}$ and $X_2$.
\begin{proof}
We will exhibit $u \in (H^1(0,T;U))^*$ and $ \varphi \in D(A^{*\infty})$ such that
\[ \forall p >2,\quad \forall 0 \leq t_0 < t_1 \leq T, \quad \langle z(\cdot) , \varphi \rangle \notin L^p(t_0,t_1).
\]

\underline{Step 1:} We first decompose $z(\cdot) = Fu$ into the sum of two elements, one of which will be proven to be a more regular function of time than the other is. Let $u \in L^2(0,T;U)$ and compute, for all $t \in [0,T]$,
\[
(Fu)(t) = \int_0^t S_{t-s}Bu(s)ds = \int_0^t (S_{t-s}-1)Bu(s)ds + \int_0^t Bu(s)ds =: (F^1u)(t) + (F^2u)(t),
\]
where we have defined two operators $F^1$ and $F^2$ which are \textit{a priori} linear and continuous from $L^2(0,T;U)$ to $C([0,T];X_{-1})$. \par 
\underline{Step 2:} We claim that actually,
\[
F^1 \in \mathcal{L}_c(\mathcal{U}_{-1} ; C_{(0)}([0,T] ; X_{-2})),\quad F^2 \in \mathcal{L}_c(\mathcal{U}_{-1} ; L^2(0,T;X_{-1})),
\]
where the notation $C_{(0)}$ was introduced in the \hyperlink{notations}{Notations}. To prove the claim, let us first deal with the first assertion. Since 
\[
F^1(L^2(0,T;U)) \subset C_{(0)}([0,T];X_{-1}) \hookrightarrow C_{(0)}([0,T];X_{-2}),
\]
and $C_{(0)}([0,T] ; X_{-2})$ is a closed subspace of $L^\infty(0,T;X_{-2})$, it is enough to show that 
\[
F^1 \in \mathcal{L}_c(\mathcal{U}_{-1} ;L^\infty(0,T;X_{-2})).
\]
By standard functional analysis arguments this boils down to establishing that
\[
\exists c > 0,\quad \forall u \in L^2(0,T;U),\quad \|F^1u\|_{L^\infty(0,T;X_{-2})} \leq c \|u\|_{\mathcal{U}_{-1} }.
\]
Because $X_{-2}$ is a Hilbert space, the identification $L^\infty(0,T;X_{-2}) = (L^1(0,T;X_2))'$ is natural  (see \cite[Section 4.1]{diestel_uhl}) and we have the dual representation of the norm:
\[
\forall f \in L^\infty(0,T;X_{-2}),\quad \| f \|_{L^\infty(0,T;X_{-2})} = \sup_{g \in L^1(0,T;X_{2}) \setminus \{ 0 \}} \frac{1}{\|g\|_{L^1(0,T;X_{2})} } \left| \int_0^T \langle f(t),g(t) \rangle_{X_{-2},X_2} dt \right|.
\]
Thus it is enough to show that
\begin{equation}\label{eq:estmiation_Phi}
\begin{split}
\exists c > 0,\quad \forall u \in L^2(0,T;U),\quad & \forall g \in L^1(0,T;X_2), \\
&\quad \left| \int_0^T \langle (F^1u)(t),g(t) \rangle_{X_{-2},X_2} dt \right| \leq c \|u\|_{\mathcal{U}_{-1} } \|g\|_{L^1(0,T;X_{2})}.
\end{split}
\end{equation}
To this aim fix such $u$ and $g$, we compute 
\[
\int_0^T \langle (F^1u)(t),g(t) \rangle_{X_{-2},X_2} dt = \int_0^T \left( u(s) ,\int_s^T B^*(S_{t-s}^*-1)g(t)dt\right)_U ds,
\]
so that defining 
\[
(\Phi g)(s) = \int_s^T B^*(S_{t-s}^*-1)g(t)dt,
\]
we are left to show that 
\[
\Phi \in \mathcal{L}_c(L^1(0,T;X_2) ; H^1(0,T;U)).
\]
It is straightforward that 
\[
\Phi \in \mathcal{L}_c(C([0,T];X_2) ;  C_{(T)}([0,T];U) \cap C^1([0,T];U)),\]
with 
\[
\frac{d}{ds}\Phi g =  - \int_s^T B^*S_{t-s}^*A^*g(t)dt.
\]
Therefore, for any $g \in C([0,T];X_2)$, we obtain the estimation 
\begin{align*}
\|\Phi g \|_{H^1(0,T;U)} &\lesssim \left \| \frac{d}{ds}\Phi g \right\|_{L^2(0,T;U)} \\
&= \left \| \int_s^T B^*S_{t-s}^*A^*g(t)dt \right\|_{L^2(0,T;U)} \\
&\lesssim \left \| \int_s^T B^*S_{t-s}^*A^*g(t)dt \right\|_{L^\infty(0,T;U)} \\
&\lesssim \| g\|_{L^1(0,T;X_2)},
\end{align*}
where the first inequality is from the Poincaré type inequality
\[
\exists C_P > 0,\quad \forall \psi \in H^1_{(T)}(0,T;U),\quad \|\psi\|_{H^1(0,T;U)} \leq C_P \|\psi'\|_{L^2(0,T;U)},
\]
and the third one was obtained in \eqref{eq:estimation_Linfty}. This shows that indeed $\Phi$ has the desired regularity, hence \eqref{eq:estmiation_Phi} holds, and thus $F^1$ has the claimed regularity. \par 
Let us now briefly explain how to obtain the second assertion, which we recall to be
\[
F^2 \in \mathcal{L}_c(\mathcal{U}_{-1} ; L^2(0,T;X_{-1})).
\]
From the computation   
\[
\int_0^T \langle (F^2u)(t),g(t) \rangle_{X_{-2},X_2} dt = \int_0^T \left( u(s) ,\int_s^T B^*g(t)dt\right)_U ds,
\]
we see that the second assertion is equivalent to the operator $\Psi$ defined by
\[
(\Psi g)(s) = \int_s^T B^*g(t)dt,
\]
to be of class $\mathcal{L}_c(L^2(0,T;X_1) ; H^1(0,T;U))$. This regularity for $\Psi$ is straightforward to check, whence the second assertion holds. \par
\underline{Step 3:} From the previous step, we are left to show that the $F^2$ component is irregular. More precisely, we will show that for some $u \in \mathcal{U}_{-1}$ and some $\varphi \in D(A^{*\infty})$,
\[ \forall p >2,\quad \forall 0 \leq t_0 < t_1 \leq T,\quad \langle F^2u , \varphi \rangle \notin L^p(t_0,t_1).
\]
We let $(u_0,\varphi_0) \in U \times D(A^{*\infty})$ be such that $\langle Bu_0 , \varphi_0 \rangle \neq 0$. This is possible because $B$ (and hence $B^*$) is assumed to be non identically zero, and $D(A^{*\infty})$ is dense in $D(A^*)$. We further introduce $\alpha \in L^2(0,T)$ such that 
\[
\forall p > 2,\quad \forall 0 \leq t_0 < t_1 \leq T,\quad \alpha \notin L^p(t_0,t_1).
\]
We then consider the control law $u = -\alpha' \otimes u_0$, that is 
\[ \forall \phi \in H^1(0,T;U),\quad
\langle u , \phi \rangle = \int_0^T \alpha(s)(u_0 , \dot{\phi}(s))_Uds,
\]
which is clearly in $(H^1(0,T;U))^*$. Assume by contradiction that for some $p > 2$ and $0 \leq t_0 < t_1 \leq T$, we have $\langle F^2u , \varphi \rangle \in L^p(t_0,t_1)$. Without loss of generality, we can also assume that $p < \infty$. Then, for all $\varphi \in D(A^{*\infty})$, there exists a constant $c = c(\varphi) > 0$ such that 
\[
\forall \psi \in C_c^\infty(t_0,t_1),\quad \left| \int_{t_0}^{t_1} \langle (F^2u)(t) , \varphi \rangle \psi(t)dt \right| \leq c \|\psi \|_{L^{p'}(t_0,t_1)}.
\]
Taking real valued $\psi$ leads to the computation 
\begin{align*}
\int_{t_0}^{t_1} \langle (F^2u)(t) , \varphi \rangle \psi(t)dt &= \int_0^T \langle (F^2u)(t),\psi(t)\varphi \rangle_{X_{-2} , X_{-2}} dt \\ 
&= \langle F^2u , \psi \otimes \varphi \rangle_{L^2(0,T;X_{-1}) , L^2(0,T;X_{1})} \\
&= \langle u , \Psi(\psi \otimes \varphi) \rangle_{\mathcal{U}_{-1},\mathcal{U}_1} \\
&= ( u , \Psi(\psi \otimes \varphi))_{L^2(0,T;U)} \\
&= \int_0^T \alpha(s) \left(u_0 , \frac{d}{ds} \int_s^T B^*\psi(t) \varphi dt \right)_U ds \\
&= \int_0^T \alpha(s) \left(u_0 , -B^*\psi(s) \varphi \right)_U ds \\
&= - \langle Bu_0 , \varphi \rangle_{X_{-1} , X_1} \int_{t_0}^{t_1} \alpha(s) \psi(s) ds. 
\end{align*}
Taking $\varphi = \varphi_0$ and invoking 
\[
\langle Bu_0 , \varphi \rangle_{X_{-1} , X_1} \neq 0,
\]
we deduce that 
\[
\left| \int_{t_0}^{t_1} \alpha(s) \psi(s) ds \right| \leq c \|\psi\|_{L^{p'}(t_0,t_1)},
\]
for all $\psi \in C_c^\infty(t_0,t_1)$, for a constant $c > 0$ not depending on $\psi$. But this estimate brings $\alpha \in L^p(t_0,t_1)$ in view of the Riesz representation Theorem (see \textit{e.g.} \cite[Theorem 4.11]{brezis}), which is the seeken contradiction.
\end{proof}	
\section{Other functional spaces}\label{sec:optimality_sobolev_tower}
\subsection{Control laws in $H^{-1}$}\label{sec:counter_h-1}
In this section, we consider the possibility of taking $u \in H^{-1}(0,T;U)$, where we recall that
\[
H^{-1}(0,T;U) := (H^1_0(0,T;U))^*.
\]
We will first propose an elementary example for which such control laws do not allow to define the generalized final state. We will then give a condition under which these control laws allow for well-defined generalized final states. 
\subsubsection{A counter-example}\label{sec:counter_example_heat}
To see what could be the obstruction to consider control laws in $H^{-1}(0,T;U)$, it is instructive to try to prove that $F_T$ has a $\mathcal{L}_c(H^{-1}(0,T;U) ; D(A^*)^*)$ extension. By duality, this boils down to showing that 
\[
\forall \varphi \in D(A^*),\quad F_T^* \varphi \in H^1_0(0,T;U),
\]
which is equivalent to 
\[
\forall \varphi \in D(A^*),\quad B^*\varphi = B^*S_T^*\varphi = 0.
\]
If that was the case, the operator $B$ would be null, so that $F_T$ cannot have the announced regularity when $B$ is not identically null. From this we learn that $D(A^*)^*$ is not a good choice of co-domain to extend $F_T$, but defining
\[
\mathcal{W} := \{ \varphi \in D(A^*) : B^* \varphi = B^*S_T^*\varphi = 0 \},
\]
one may expect $F_T^* \in \mathcal{L}_c(\mathcal{W} ; H^1_0(0,T;U))$ so that $F_T \in \mathcal{L}_c(H^{-1}(0,T;U) ; \mathcal{W}')$. Note carefully that this makes sense when $\mathcal{W}'$ is a super-space of $X$, and the only reasonable way to expect that is to require the inclusion $\mathcal{W} \subset X$ to be dense. However, the constraint $B^*S_T^*\varphi = 0$ cannot be dense in $X$ when the two conditions
\[
B^*S_T^* \in \mathcal{L}_c(X;U),\quad \mathrm{and} \quad B^*S_T^* \neq 0,
\]
are satisfied. The first above condition may occur when $S_T^*$ is regularizing. \par 
Let us now consider the heat equation on a bounded interval that is controlled at one end by the Neuman action:
\[
\left\lbrace \begin{array}{cccc}
z_t &=& z_{xx}, & 0 < x < \pi, \\
z_x(t,0) &=& u(t),\\
z_x(t,\pi) &=& 0.
\end{array}\right.
\]
Following \cite[Section 10.2.1]{tucwei}, we represent this system by the LTI system $\Sigma(A,B)$ where
\[
X = L^2(0,\pi),\quad U = \mathbb{C},
\]
and
\[
Az = z_{xx},\quad D(A) = \{ z \in H^2(0,\pi) : z_x(0) = z_x(\pi) = 0, \},\quad B^* = - \delta_0.
\]
Let $\mathcal{W}$ be a topological vector space such that the inclusion $\mathcal{W} \subset L^2(0,\pi)$ is dense and continuous. For then $X$ naturally embeds into $\mathcal{W}'$ through $z \mapsto (z,\cdot)_X$, when $\mathcal{W}'$ is endowed of the weak-$*$ topology. Let us denote $\mathcal{W}'_{w-*}$ the corresponding topological vector space and assume that $z_0$ is taken to be zero.
\begin{Proposition}
For all topological vector space $\mathcal{W}$ as above, the map $u \mapsto z(T)$ has no linear and continuous extension from $H^{-1}(0,T)$ to $\mathcal{W}'_{w-*}$.
\end{Proposition}
\begin{proof}
We claim that the operator $\delta_{t=0} F_T^* \in \mathcal{L}_c(D(A^*) ; \mathbb{C})$ extends as a linear form on $X$. To see this, compute for $\varphi \in D(A^*)$, 
\[
\delta_{t=0} F_T^*\varphi =  B^*S_T^*\varphi = - \delta_{x=0} S_T^* \varphi = - \sum_{n=0}^\infty e^{\lambda_n T} (\varphi , c_n)_X c_n(0) = \left( \varphi ,  - \sum_{n=0}^\infty e^{\lambda_n T} c_n(0)c_n \right)_X,
\]
where we have denoted, for all $n \in \mathbb{N}$,
\[
\lambda_n = - n^2,
\]
the $n$-th eigenvalue of $A$, and 
\[
c_n(x) = \left\lbrace \begin{array}{ccc}
\sqrt{\frac{2}{\pi}} \cos(nx) & \mathrm{if} & n \in \mathbb{N}^* \\
\frac{1}{\sqrt{\pi}} & \mathrm{if} & n =0 
\end{array}\right. ,
\]
an associated eigenvector, so that $(c_n)_{n=0}^\infty$ forms a Hilbert basis of $X$. Denote
\[
\psi(x) := - \sum_{n=0}^\infty e^{\lambda_n T} c_n(0)c_n(x),
\]
which is clearly an $L^2(0,\pi)$ function, hence in $X$, so that 
\[
\delta_{t=0} F_T^*\varphi = (\varphi , \psi)_X,\quad \forall \varphi \in X,
\]
yields the announced extension. \par 
Assume by contradiction that $F_T$ has a $\mathcal{L}_c(H^{-1}(0,T) ; \mathcal{W}'_{w-*})$ extension. Then, reasoning by duality we obtain that $F_T^*$ maps $\mathcal{W}$ into $H^1_0(0,T)$. Thus, the linear form $\delta_{t=0} F_T^*$ is identically zero on $\mathcal{W}$, so that 
\[
\mathcal{W} \subset \opker \delta_{t=0}F_T^* = \{ \psi \}^\perp.
\]
Because $\psi$ is not identically zero, this negates the density of $\mathcal{W}$ in $X$, which is the seeken contradiction. 
\end{proof}
\subsubsection{A positive result}
Let us recall that the set $\mathcal{W}$ is defined by
\[
\mathcal{W} = \{ \varphi \in D(A^*) : B^* \varphi = B^*S_T^*\varphi = 0 \}.
\]
When endowed with the $D(A^*)$-norm, it becomes a Hilbert space. If $\mathcal{W}$ is a dense subset of $X$ we denote $\mathcal{W}^*$ its anti-dual space with respect to the pivot $X$.
\begin{Theorem}\label{theo:endpoint_H-1}
Assume that $\mathcal{W}$ is dense in $X$. Then the map $u \mapsto z(T)$ has a unique linear and continuous extension from $H^{-1}(0,T;U)$ to $\mathcal{W}^*$.
\end{Theorem}
\begin{proof}
By duality, it boils down to check that $F_T^*$ maps $\mathcal{W}$ to $H^1_0(0,T;U)$, which is trivial.
\end{proof}
\begin{Example} Consider the following one-dimensional wave equation that is controlled at one end through the Neumann action
\begin{equation}\label{eq:wave_controle_H-1}
\left\lbrace \begin{array}{cccc}
z_{tt} &=& z_{xx}, & 0 < x < \pi,\\
z_x(t,0) &=& u(t),\\
z(t,\pi) &=& 0.
\end{array}\right.
\end{equation}
Following \cite[Section 10.2.2]{tucwei} we model this system as the LTI system $\Sigma(A,B)$ defined by 
\[
X = H^1_{(\pi)}(0,\pi)\times  L^2(0,\pi),\quad A = \left( \begin{array}{cc}
0 & 1 \\
\frac{d^2}{dx^2} & 0 \\
\end{array}\right),\quad U = \mathbb{C},
\]
and
\[
D(A) = \{ (u,v) \in H^2(0,\pi) \times H^1(0,\pi) : u(\pi) = u_x(0) = v(\pi)  = 0 \},\quad B^*(\varphi , \psi) = - \psi(0).
\]
We claim that the above system falls into the realm of Theorem \eqref{theo:endpoint_H-1} above. To check this, we have to show that the set $\mathcal{W}$ defined above is dense in $X$. We will do so using the method of the characteristics. We acknowledge that $A$ is skew-adjoint, so that $D(A^*) = D(A)$ and $A^* = -A$. Let $(\varphi,\psi) \in D(A^*)$, the two conditions for $(\varphi,\psi)$ to lie in $\mathcal{W}$ are that 
\[
B^*(\varphi, \psi) = 0,\quad \mathrm{and} \quad B^*S_T^*(\varphi,\psi) = 0.
\]
The first condition amounts to 
\[
\psi(0) = 0,
\]
and the second amounts to 
\begin{equation}\label{eq:charac_w}
w_t(T,0) = 0,
\end{equation}
where $w$ is the solution of 
\[
\left\lbrace \begin{array}{ccccc}
w_{tt}&=&w_{xx},& 0 < t < T,& 0< x < \pi,\\
w_x(t,0) &=& 0,\\
w(t,\pi) &=& 0,\\
w(0,x) &=& \varphi(x),\\
w_t(0,x) &=& - \psi(x). 
\end{array}\right.
\]
To fix the ideas, assume for a moment that $0 < T \leq \pi$, we introduce the so-called Riemann invariants
\[
\xi = w_t-w_x,\quad \eta = w_t+w_x,
\]
so that $\xi$ and $\eta$ respectively satisfy the transport to the right and the transport to the left:
\[
\left\lbrace \begin{array}{ccc}
\xi_t + \xi_x &=& 0,\\
\xi(t,0) &=& \eta(t,0),\\
\end{array} \right.,\quad \left\lbrace \begin{array}{ccc}
\eta_t-\eta_x &=& 0, \\
\eta(t,\pi) &=& -\xi(t,\pi).
\end{array}\right.
\]
The condition \eqref{eq:charac_w} then writes as 
\[
(\xi + \eta)(T,0) = 0,
\]
so that given the boundary condition of $\xi$, it translates as 
\[
\xi(T,0) = \eta(T,0) = 0.
\]
Note that $\xi$ is transported to the right, and that its boundary condition on the left is given by $\eta$, from which de deduce that \eqref{eq:charac_w} is nothing else than 
\[
\eta(T,0) = 0.
\]
Using that $\eta$ solves the transport to the left we see that this is equivalent to 
\[
-\psi(T) + \varphi_x(T) = 0.
\]
Therefore, 
\[
\mathcal{W} = \{ (\varphi,\psi) \in D(A^*) : \psi(0) = -\psi(T) + \varphi_x(T) = 0 \},
\]
and to check that this set is dense in $X$, it is enough to verify that the smaller set 
\[
\tilde{\mathcal{W}} :=  \{ (\varphi,\psi) \in D(A^*) : \psi(0) = \psi(T) = \varphi_x(T) = 0 \},
\]
is dense in $X$. Separating the variables, this boils down to verifying that 
\[
\{ \varphi \in H^2(0,\pi) : \varphi(\pi) = \varphi_x(0) = \varphi_x(T) = 0  \},
\]
is dense in $H^1_{(\pi)}(0,\pi)$, and that 
\[
\{ \psi \in H^1(0,\pi) : \psi(0) = \psi(T) = 0 \},
\]
is dense in $L^2(0,\pi)$. The proofs of these two facts are elementary, we omit them and we refer the interested reader to \cite[Lemme 11.1]{Lions_Magenes} or \cite[Lemma 17.3]{tartar}. This shows that $\mathcal{W}$ is dense in $X$, so that Theorem \eqref{theo:endpoint_H-1} applies, assuming that $T \leq \pi$. In case $T > 0$ is arbitrary one needs to propagate backward in time the condition 
\[
\eta(T,0) = 0,
\]
along caracteristic rays, making them successively reflect with respect to the segments $[0,T] \times \{ 0\}$ and $[0,T] \times \{\pi \}$. Therefore, the condition 
\[
\eta(T,0) = 0,
\]
either takes one of the forms 
\[
\eta(0,\alpha) = 0,\quad \mathrm{or} \quad \xi(0,\alpha) = 0,
\]
for some $\alpha \in [0,\pi]$. The density of $\tilde{\mathcal{W}}$ (with $\alpha$ in place of $T$) still brings the result, because 
\[
\xi(0,x) = -\psi(x) - \varphi(x),\quad \eta(0,x) = -\psi(x) + \varphi(x),
\]
for all $x \in [0,\pi]$.
\end{Example}
\subsection{Time regularity}
In Proposition \eqref{Prop:optimal_reg_N=1} we have shown that in full generality, for control laws $u$ in $\mathcal{U}_{-1}$, the state curve is $L^2$ in time without any possible improvement. Here we describe a general situation where this can be improved, in a different functional setting. We define for arbitrary $k \in \mathbb{N}$ the space 
\[
W_k = \{ \varphi \in D(A^{*k}) : B^*\varphi = \cdots = B^*A^{*(k-1)} \varphi = 0 \} = \bigcap_{i=0}^{k-1} \opker B^*A^{*i}.
\]
If $W_k$ is a dense subset of $X$, we can consider $W_k^*$ its anti-dual space with respect to $X$, that we will denote $W_{-k}$. 
\begin{Proposition}\label{prop:time_reg_arbitrary}
Let $N,M \in \mathbb{N}$ arbitrary and assume that $W_{N+M}$ is a dense subset of $X$. Then, the map $u \mapsto z(\cdot)$ extends linearly and continuously from $\mathcal{U}_{-M}$ to $C^N([0,T] ; W_{-N-M})$.
\end{Proposition}
\begin{proof}
The proof is done by induction on $N \in \mathbb{N}$. Assume first that $N = 0$, we have to show that for any $M \in \mathbb{N}$, the map 
\[
F \in \mathcal{L}_c(L^2(0,T;U) ; C([0,T];X)),\quad (Fu)(t) = \int_0^t S_{t-s}Bu(s)ds,
\]
extends as a bounded operator from $\mathcal{U}_{-M}$ to $C([0,T] ; W_{-M})$. Fix $M \in \mathbb{N}$, we let $u \in L^2(0,T;U)$ and fix $\varphi \in W_M$. We compute, for all $t \in [0,T]$,
\[
((Fu)(t),\varphi)_X = \int_0^t (u(s) , B^*S_{t-s}^*\varphi)_u ds = \int_0^T (u(s) , 1_{[0,t]}(s)B^*S_{t-s}^*\varphi)_u ds.
\]
Observe that because $\varphi \in D(A^{*M})$, the function 
\[
s \mapsto B^*S_{t-s}^*\varphi,
\]
is of class $H^M(0,t;U)$. We further claim that extending it by $0$ on $(t,T)$ makes it a $H^M$ function. Indeed, because the interval $(0,T)$ is one dimensional it is enough to check that 
\[
\forall k = 0, ..., M-1,\quad \frac{d^k}{ds^k} B^*S_{t-s}^* \varphi \xrightarrow[s \rightarrow t^-]{U} 0.
\] 
But this follows from the definition of $W_{-M}$, owing to the computation 
\[
\frac{d^k}{ds^k} B^*S_{t-s}^* \varphi = B^*S_{t-s}^* (-A^*)^k \varphi \xrightarrow[s \rightarrow t^-]{U} (-1)^k B^* A^{*k} \varphi = 0.
\]
We therefore infer that 
\[
\left[ s \mapsto 1_{[0,t]}(s)B^*S_{t-s}^*\varphi \right] \in H^M(0,T;U).
\]
To complete the proof of the initialization of the induction, we are left to find a bound 
\[
\| 1_{[0,t]}(\cdot)B^*S_{t-\cdot}^*\varphi \|_{H^M(0,T;U)} \leq C \| \varphi\|_{D(A^{*M})},
\]
with $C$ not depending on $\varphi$ nor on $t$. Going back to the first Step of the proof of Proposition \eqref{prop:extension_F_T}, we see that 
\begin{align*}
\| 1_{[0,t]}(\cdot)B^*S_{t-\cdot}^*\varphi \|_{H^M(0,T;U)} &= \| B^*S_{t-\cdot}^*\varphi \|_{H^M(0,t;U)} \\
&= \|F_t^* \varphi\|_{H^M(0,t;U)} \\
&\leq C_{\operatorname{adm}}(t)C(M)\|\varphi\|_{D(A^{*M})} \\
&\leq C_{\operatorname{adm}}(T)C(M)\|\varphi\|_{D(A^{*M})},
\end{align*}
where 
\[
C_{\operatorname{adm}}(t) := \sup\left\lbrace \frac{\|S_t^*\varphi\|_{X}}{\|F_t^*\varphi\|_{L^2(0,t;U)}} : \varphi \in D(A^*),\quad F_t^*\varphi \neq 0 \right\rbrace,
\]
is the optimal admissibility constant for the system $\Sigma(A,B)$ at time $t$, and $C(M)$ is a universal constant only depending on $M$. This completes the initialization of the induction. \par 
We now let $N \in \mathbb{N}$ be arbitrary such that for all $M \in \mathbb{N}$, the claimed regularity holds. We let $M \in \mathbb{N}$ fixed, assume that $W_{N+M+1}$ is a dense subset of $X$, and show that 
\[
F \in \mathcal{L}_c( \mathcal{U}_{-M} ; C^{N+1}([0,T] ; W_{-N-M-1})).
\] 
We let $u \in L^2(0,T;U)$, note that in view of the induction hypothesis we have 
\[
Fu \in C^N([0,T] ; W_{-N-M}).
\]
We compute, for arbitrary $\varphi \in W_{N+M+1}$, in the sense of distributions,
\begin{align*}
\frac{d}{dt} ((Fu)(t) , \varphi)_X &= \frac{d}{dt} \int_0^t (u(s),B^*S_{t-s}^*\varphi)_U ds  \\
&= \int_0^t (u(s) , B^*S_{t-s}^*A^*\varphi)_U ds + (u(t) , B^*\varphi)_U\\
&= \int_0^t (u(s) , B^*S_{t-s}^*A^*\varphi)_U ds \\
&= \left( \int_0^t S_{t-s} B u(s) ds , A^*\varphi \right)_X \\
&= \left( A \int_0^t S_{t-s} B u(s) ds , \varphi \right)_X,
\end{align*}
where the second equality is the Leibniz rule for the differentiation of parametrized integrals. From density of $W_{N+M+1}$ in $X$, we deduce that 
\[
\frac{d}{dt} Fu = A \int_0^t S_{t-s} B u(s) ds,
\]
for all $t \in [0,T]$. Differentiating again $N-1$ times, we arrive to 
\[
\frac{d^N}{dt^N} Fu = A^N \int_0^t S_{t-s} B u(s) ds.
\]
We are left to show that the above right-hand side is a $C^1([0,T] ; W_{N-M-1})$ function, with an estimate that does not depend on $u$. To do so, observe that the same scalar derivation as above yields
\[
\frac{d}{dt} A^N \int_0^t S_{t-s} B u(s) ds = A^{N+1} \int_0^t S_{t-s} B u(s) ds \quad \mathrm{in} \quad \mathcal{D}'(0,T;W_{-N-M-1}).
\]
This ends the proof as the $C([0,T] ; W_{-N-M-1})$-norm of the right-hand side above is controlled by the $\mathcal{U}_{-M}$-norm of $u$, owing to the computations performed for the case $N=0$.
\end{proof}
\section{Duality between controllability and observability}
From Example \eqref{ex:toy_model} we see that the right object to consider for controllability is the generalized final state, which is the extension of the operator $\Xi_T$, still denoted by the same symbol. This calls for a definition of what we mean by ``controllability" when the controls are less regular than $L^2$. Fix $N,M \in \mathbb{Z}$ arbitrary, $\nu := - \min(0,N,M)$, and acknowledge that $\nu \geq 0$ is such that the final state operator $\Xi_T$ maps $X_N \times \mathcal{U}_M$ to $X_{-\nu}$. Let us introduce
\begin{itemize}
\item The space of initial conditions: $\mathcal{X}$ is a closed linear subspace of $X_N$ for some $N \in \mathbb{Z}$, and $P : X_N \rightarrow X_N$ a linear and continuous projection onto $\mathcal{X}$; 
\item The space of control laws: $\mathcal{U}$ is a closed linear subspace of $\mathcal{U}_M$ for some $M \in \mathbb{Z}$ and $Q : \mathcal{U}_M \rightarrow \mathcal{U}_M$ a linear and continuous projection onto $\mathcal{U}$;
\item The output operator: $C \in \mathcal{L}_c(X_{-\nu} ; Y)$, where $Y$ is a Hilbert space.
\end{itemize}
\begin{Definition}\label{def:NC}
The system $\Sigma(A,B,C)$ is null controllable (at time $T$), with initial conditions in $\mathcal{X}$ and control laws in $\mathcal{U}$ if 
\[
\forall z_0 \in \mathcal{X},\quad \exists u \in \mathcal{U},\quad C\Xi_T (z_0,u)=0.
\]
\end{Definition}
We let $Y^*$ be a realization of the anti-dual space of $Y$ and $C^* \in \mathcal{L}_c(Y^* ; X_{\nu})$ the corresponding adjoint operator.  Let us also denote $P^* \in \mathcal{L}_c(X_{-N})$ the adjoint operator of $P$ when $(X_N)^* = X_{-N}$, similarly $Q^* \in \mathcal{L}_c(\mathcal{U}_{-M})$, and $S_T^*$ the usual adjoint of $S_T$ (\textit{i.e.} when $X$ is a pivot space). 
\begin{Proposition}\label{prop:contro_general_rigorous}
The system $\Sigma(A,B,C)$ is null controllable (at time $T$), with initial conditions in $\mathcal{X}$ and control laws in $\mathcal{U}$ if and only if there is a constant $c >0$ such that for all $y^* \in Y^*$ there holds 
\[
\| P^* S_T^* C^* y^* \|_{X_{-N}} \leq c  \| Q^* B^* S_t^* C^* y^* \|_{\mathcal{U}_{-M}}.
\]
\end{Proposition}
\begin{proof}
Assume that the system is indeed null controlable. Then we obtain by definition
\[
\forall z_0 \in \mathcal{X},\quad \exists u \in \mathcal{U},\quad CS_Tz_0 + CF_Tu = 0.
\]
Therefore, 
\[
CS_T(\mathcal{X}) \subset CF_T (\mathcal{U}),
\]
and because $P$ and $Q$ are surjective maps this re-writes 
\[
CS_TP(X_N) \subset CF_TQ(\mathcal{U}_M).
\]
We interpret this inclusion as the ranges inclusion  
\[
\opRange L \subset \opRange R,
\]
where we have defined the operators
\[
L : \left\lbrace \begin{array}{ccc}
X_N & \longrightarrow & Y \\
z & \longmapsto & CS_TPz 
\end{array}\right. ,\quad
R : \left\lbrace \begin{array}{ccc}
\mathcal{U}_{M} & \longrightarrow & Y \\
u & \longmapsto & CF_TQu 
\end{array}\right. .
\]
Observe that $L$ is well-defined, linear and continuous. Indeed if $N \leq 0$ it writes as the composition 
\[
X_N \overset{P}{\longrightarrow} X_N \overset{S_T}{\longrightarrow} X_N \overset{i}{\hookrightarrow} X_{-\nu} \overset{C}{\longrightarrow} Y
\]
where $\overset{i}{\hookrightarrow}$ stands for a continuous injection. In case $N > 0$, $L$ writes as 
\[
X_N \overset{P}{\longrightarrow} X_N \overset{i}{\hookrightarrow} X \overset{S_T}{\longrightarrow} X \overset{i}{\hookrightarrow} X_{-\nu} \overset{C}{\longrightarrow} Y.
\]
Similarly we obtain that $R$ is well-defined linear and continuous.  \par  
We then apply the Douglas Lemma (see \cite[Proof of Theorem 2.44]{Coron}) to find $c>0$ such that 
\[
\forall y^* \in C^*,\quad \| L^* y^* \|_{X_{-N}} \leq c \| R^*y^*\|_{\mathcal{U}_{-M}}.
\]
To obtain the desired observability inequality, we are left to show that 
\[
\forall y^* \in Y^*,\quad L^* y^* = P^*S_T^*C^*y^*,
\]
and 
\[
\forall y^* \in Y^*,\quad R^* y^* = Q^*B^*S_{T-\cdot}^*C^*y^*.
\]
Let us proof the first equality, we begin by making formal computation to be justified carefully later on. We let $y^* \in Y^*$ and $\varphi \in X_\nu \cap X_N$, we compute
\begin{align*}
\langle L^* y^* , \varphi \rangle_{X_{-N} , X_N} &= \langle y^* , L \varphi \rangle_{Y^* , Y} \\
&= \langle y^* , CS_TP\varphi \rangle_{Y^*,Y} \\
&= \langle C^* y^* , S_T P \varphi \rangle_{X_\nu , X_{-\nu}} \\
&= \overline{\langle S_T P \varphi , C^* y^* \rangle_{X_{-\nu} , X_\nu}} \\
&= \overline{\langle  P \varphi ,S_T^* C^* y^* \rangle_{X_{-\nu} ,X_\nu}} \numberthis \label{eq:justif_dual_1} \\
&= \overline{\langle  P \varphi , S_T^*C^* y^* \rangle_{X_N , X_{-N}}} \numberthis \label{eq:justif_dual_2} \\
&= \overline{\langle \varphi , P^* S_T^* C^*y^* \rangle_{X_N , X_{-N}}} \\
&= \langle P^* S_T^* C^* y^* , \varphi \rangle_{X_{-N} , X_N},\numberthis \label{eq:justif_dual_3}
\end{align*}
where the only equalities to justify are \eqref{eq:justif_dual_1}, \eqref{eq:justif_dual_2} and \eqref{eq:justif_dual_3}. For \eqref{eq:justif_dual_1} we invoke the fact that the operator 
\[
\left\lbrace \begin{array}{ccc}
X_{-\nu} & \longrightarrow & X_{-\nu} \\
z & \longmapsto & S_Tz
\end{array}\right.
\]
has adjoint 
\[
\left\lbrace \begin{array}{ccc}
X_{\nu} & \longrightarrow & X_{\nu} \\
\varphi & \longmapsto & S_T^* \varphi
\end{array}\right. .
\]
For \eqref{eq:justif_dual_2}, we invoke 
\[
P \varphi \in X_{-\nu} \cap X_N,\quad S_T^*C^*y^* \in X_\nu \cap X_{-N},
\]
and the following symetrized compatibility of the successive duality brackets
\[
\forall N_1,N_2 \in \mathbb{Z},\quad \forall z \in X_{-\min(N_1,N_2)},\quad \forall \varphi \in X_{\max(N_1,N_2)},\quad \langle z , \varphi \rangle_{-N_1,N_1} = \langle z , \varphi \rangle_{-N_2,N_2}.
\]
For \eqref{eq:justif_dual_3} we invoke the following property
\[
\forall N \in \mathbb{Z},\quad \forall z \in X_{-N},\quad \forall \varphi \in X_N,\quad \overline{\langle z , \varphi \rangle_{X_{-N} , X_N}} = \langle \varphi , z \rangle_{X_N , X_{-N}}.
\]
The proof of the second equality is very similar and we omit it. \par 
Therefore, if the system is null controllable we obtain the above admissibility inequality. The proof of the converse assertion can be done using all the given arguments, since the Douglas Lemma is a logical equivalence. 
\end{proof}
\section{Application to a fluid-structure model}\label{Section:application}
\subsection{Presentation of the model}
In \cite{Zhang_Zuazua} the authors investigate the null controllability of the fluid-structure model \eqref{eq_zhang_zuazua} that we recall here :
\[
\left\lbrace \begin{array}{cccc}
u_t &=& u_{xx}, & 0 < x < 1 ,\\
v_{tt} &=& v_{xx}, & -1 < x < 0,\\
u(t,1) &=& g_1(t), \\
v(t,-1) &=& 0, \\
u(t,0) &=& v(t,0), \\
u_x(t,0) &=& v_x(t,0), \\
\end{array}\right.
\]
where $g_1$ is the control. Prior to state and prove our result we sum up the spectral analysis performed in \cite{Zhang_Zuazua}. The control system \eqref{eq_zhang_zuazua} is realized as an LTI system $\Sigma(A,B)$ defined on the state space 
\[
X = \left\lbrace 
\left( \begin{array}{c}
u \\
v \\
v_t 
\end{array} \right) \in H^1(0,1) \times H^1(-1,0) \times L^2(-1,0) : u(1) = 0,\quad v(-1) = 0,\quad u(0) = v(0) \right\rbrace .
\]
The control space is $U = \mathbb{C}$ and the adjoint of $A$ is defined as follows. It has domain 
\[
D(A^*) = \left\lbrace 
\left( \begin{array}{c}
f \\
g \\
h 
\end{array} \right) \in H^3(0,1) \times H^2(-1,0) \times H^1(-1,0) : \left \vert \begin{array}{ccccc}
f(1) &=& 0 \\
g(-1) &=& 0 \\
f(0) &=& g(0) \\
f_x(0) &=& g_x(0) \\
f_{xx}(1) &=& h(-1) &=& 0 \\
f_{xx}(0) &=& h(0)
\end{array}\right.  \right\rbrace ,
\]
and is defined algebraically via 
\[
\forall \left( \begin{array}{c}
f \\
g \\
h 
\end{array} \right) \in D(A^*),\quad A^* \left( \begin{array}{c}
f \\
g \\
h 
\end{array} \right) = \left( \begin{array}{c}
f_{xx} \\
h \\
g_{xx} 
\end{array} \right).
\]
Note that actually, interpreting \eqref{eq_zhang_zuazua} \textit{verbatim} one could expect 
\[A^* \left( \begin{array}{c}
f \\
g \\
h 
\end{array} \right) = \left( \begin{array}{c}
f_{xx} \\
-h \\
-g_{xx} 
\end{array} \right),
\]
but this is a harmless simplification which does not affect what follows. Further, the control operator is given by 
\[
\forall \left( \begin{array}{c}
f \\
g \\
h 
\end{array} \right) \in D(A^*),\quad B^* \left( \begin{array}{c}
f \\
g \\
h 
\end{array} \right)  = f_x(1).
\]
It turns out that this operator $A^*$ can be reduced in Jordan block in the following sense. There is a Riesz basis of $X$ that writes
\begin{equation}\label{eq_riesz_A*}
\cup_{j=1}^{n_0} \{ \varphi_{j,0} , ..., \varphi_{j,m_j-1}  \} \bigcup \{ \varphi_\ell^p \}_{\ell = \tilde{\ell}_1}^\infty \bigcup \{ \varphi_k^h \}_{|k| \geq \tilde{k}_1}^\infty,
\end{equation}
where 
\begin{itemize}
\item $n_0,\tilde{\ell}_1,\tilde{k}_1 \in \mathbb{N^*}$ and $m_0, ... ,m_{n_0-1} \in \mathbb{N}^*$;
\item For all $j = 1 , ... , n_0$, the vector $\varphi_{j,0}$ is an eigenvector of $A^*$ of algebraic multiplicity $m_j$ and $\{ \varphi_{j,0} , ..., \varphi_{j,m_j-1}  \}$ is the associated Jordan chain;
\item For all $\ell \geq \tilde{\ell}_1$, the vector $\varphi_\ell^p$ is a normalized eigenvector of $A^*$, denote $\lambda_\ell^p$ the corresponding eigenvalue;
\item For all $|k| \geq \tilde{k}_1$, the vector $\varphi_k^h$ is a normalized eigenvector of $A^*$, denote $\lambda_k^h$ the corresponding eigenvalue.
\end{itemize}
In the above, the superscript ``$p$" (resp. ``$h$") stand for ``parabolic" (resp. ``hyperbolic"). This denomination is because the branch of eigenvalues $(\lambda_\ell^p)_{\ell \geq \tilde{\ell}_1}$ (resp. $(\lambda_k^h)_{|k| \geq \tilde{k}_1}$) is asymptotically close to the eigenvalues of the heat (resp. wave) sub-system of \eqref{eq_zhang_zuazua}. \par 
Concerning the hyperbolic eigenvalues, we have the following asymptotics as $|k| \rightarrow \infty$, which is still from \cite{Zhang_Zuazua} :
\begin{equation}\label{eq_asymp_eigen}
\lambda_k^h = - \frac{1}{\sqrt{|1+2k|\pi}} + \left( \frac{1}{2} + k \right)\pi i  + \frac{\opsgn(k)}{\sqrt{|1+2k|\pi}}i + O\left( \frac{1}{|k|} \right).
\end{equation}
Moreover, we have another asymptotics from \cite{Zhang_Zuazua}, this one for the values of $B^*$ against the hyperbolic eigenvectors: for any $N \in \mathbb{N}$ there exists $c > 0$ such that 
\begin{equation}\label{eq_asymp_sorie}
\forall |k| \geq \tilde{k}_1,\quad \| B^*S_t^* \varphi_k^h \|_{H^N(0,T)} \leq c e^{-\sqrt{|k|}}.
\end{equation}
\subsection{Proof of the result}
Prior to prove Proposition \eqref{prop:zz_nc} let us collect an elementary fact and introduce some notations. Let us introduce the bi-orthogonal sequence of \eqref{eq_riesz_A*}
\[
\cup_{j=1}^{n_0} \{ z_{j,0} , ..., \varphi_{j,m_j-1}  \} \bigcup \{ z_\ell^p \}_{\ell = \tilde{\ell}_1}^\infty \bigcup \{ z_k^h \}_{|k| \geq \tilde{k}_1}^\infty.
\]
Because for all $|k| \geq \tilde{k}_1$, the vector $\varphi_k^h$ is an eigenvector of $A^*$ associated to the eigenvalue $\lambda_k^h$, it is automatic that $z_k^h$ is an eigenvector of $A$ associated to the eigenvalue $\overline{\lambda_k^h}$, see \cite[Lemma 3.2.5]{curtain_zwart}. This allows for the estimation
\begin{align*}
\|\varphi_k^h \|_{D(A^N)^*} &= \sup_{\|\varphi\|_{D(A^N)} = 1} | (\varphi_k^h,\varphi)_X| \\
&\geq \left| \left( \varphi_k^h ,  \frac{z_k^h}{\| z_k^h \|_{D(A^N)}} \right)_X \right| \\
&\gtrsim \frac{1}{\sqrt{1+ |\lambda_k^h|^{2N}}}. \numberthis \label{eq:eqtimate_DAN_hyp_eigen}
\end{align*}
Further consider the subspaces of $X$ defined by
\[
X_h := \overline{\opSpan}\{ z_k^h : |k| \geq \tilde{k}_1\},\quad X_p := \overline{\opSpan}\{ z_\ell^p : \ell \geq \tilde{\ell}_1 \} \oplus \opSpan \left(\cup_{j=1}^{n_0} \{ z_{j,0} , ..., z_{j,m_j-1} \} \right),
\]
which are such that $X = X_h \oplus X_p$, the sum being not perpendicular. Consider $P_h$ the projection of $X$ onto $X_h$ parallel to $X_p$, which writes as
\[
P_hz = \sum_{|k| \geq \tilde{k}_1} (z,\varphi_k^h) z_k^h,
\]
where $(\cdot,\cdot)$ stands from now on for the scalar product of $X$. 
\begin{proof}[Proof of Proposition \eqref{prop:zz_nc}] Let us first check that the statement of Proposition \eqref{prop:zz_nc} falls into the realm of Definition \eqref{def:NC}. This boils down to check that for some Hilbert space $Y$, the operator $P_h$ is $\mathcal{L}_c(X_{-N} ; Y)$. From \cite[Proposition 2.9.3]{tucwei} we know that the latter holds with $Y = X_{-N}$ as soon as $P_h^* X_N \subset X_N$. This can be readily seen computing
\[
P_h^*\varphi = \sum_{|k| \geq \tilde{k}_1} (\varphi,z_k^h) \varphi_k^h,
\]
for all $\varphi \in X$. \par 
Therefore the statement of the Proposition makes sense, but we cannot apply Proposition \eqref{prop:contro_general_rigorous} \textit{verbatim} because $X_h \cap D(A^N)$ is not a closed subspace of a space of the form $X_M$ for some $M \in \mathbb{Z}$. Nonetheless we claim that $\Sigma(A,B)$ is null controllable (with output operator and inputs as in the statement of the Proposition \eqref{prop:zz_nc}) if and only if the following observability inequality holds 
\begin{equation}\label{eq:obs_ineq_zuazua}
\| P_h^*S_T^*P_h^* \varphi \|_{D(A^N)^*} \leq c \| B^* S_t^* P_h^* \varphi\|_{H^N(0,T)}, \quad \varphi \in D(A^{*N}).
\end{equation}
Indeed, note that null controllability is equivalent to 
\[
\forall z_0 \in X_h \cap D(A^N),\quad \exists u \in \mathcal{U}_{-N},\quad P_h S_Tz_0 + P_h F_Tu = 0.
\]
It is elementary to verify that 
\[
X_h \cap D(A^N) = P_h D(A^N),
\]
hence the above assertion is itself equivalent to 
\[
\opRange P_h S_T P_h \subset \opRange P_h F_T.
\]
It is then straightforward to adapt the proof of Proposition \eqref{prop:contro_general_rigorous} to obtain that null controllability is characterized by \eqref{eq:obs_ineq_zuazua}. \par 
To show that \eqref{eq:obs_ineq_zuazua} cannot hold with a finite $c$, we test it against the hyperbolic eigenvectors $(\varphi_k^h)_{|k| \geq \tilde{k}_1}$, and show that
\[
\frac{\left\|B^*S_t^*P_h^*\varphi_k^h\right\|_{H^N(0,T)}}{\left\| P_h^* S_T^* P_h^* \varphi_k^h \right\|_{D(A^N)^*}} \xrightarrow[k \rightarrow \infty]{} 0.
\]
We can estimate 
\begin{align*}
\left\| P_h^* S_T^* P_h^* \varphi_k^h \right\|_{D(A^N)^*} &= \left\| S_T^* \varphi_k^h \right\|_{D(A^N)^*} \\
&= \left\| e^{T \lambda_k^h} \varphi_k^h \right\|_{D(A^N)^*} \\
&= e^{T \Re \lambda_k^h}\left\| \varphi_k^h \right\|_{D(A^N)^*} \\
&\gtrsim e^{T \Re \lambda_k^h} \frac{1}{\sqrt{1+ |\lambda_k^h|^{2N}}} \\
&\geq \frac{c}{\sqrt{1+ |\lambda_k^h|^{2N}}},
\end{align*}
for a constant $c > 0$ not depending on $|k| \geq \tilde{k}_1$, where the first inequality is from \eqref{eq:eqtimate_DAN_hyp_eigen} and the second inequality is from $(\Re \lambda_k^h)_{|k| \geq \tilde{k}_1}$ being bounded (see \eqref{eq_asymp_eigen}). Now invoke the estimate \eqref{eq_asymp_sorie} to obtain, for a constant $c > 0$ independent of $|k| \geq \tilde{k}_1$, that
\[
\frac{\left\|B^*S_t^*P_h^*\varphi_k^h\right\|_{H^N(0,T)}}{\left\| P_h^* S_T^* P_h^* \varphi_k^h \right\|_{D(A^N)^*}} \leq c e^{-\sqrt{|k|}} \sqrt{1+ |\lambda_k^h|^{2N}}.
\]
We conclude by invoking \eqref{eq_asymp_eigen}, which yields $(\lambda_k^h)_{|k| \geq \tilde{k}_1}$ having a polynomial growth.
\end{proof}
\section*{Acknowledgment}
The author is indebted to Olivier Glass, Pierre Lissy and Swann Marx for their numerous, useful comments and insights. The author would also like to thank Yubo Bai for the reference \cite{lions_contro}.
\printbibliography[heading = bibintoc]

\end{document}